\documentclass{article}
\usepackage{amsmath}
\usepackage{epsfig,amsmath,amsfonts}
\usepackage{algorithmic,algorithm}

\newtheorem{theorem}{Theorem}[section]

\newtheorem{lemma}[theorem]{Lemma}

\newtheorem{remark}{Remark}

\def\be{\begin{equation}}
\def\ee{\end{equation}}
\def\bea{\begin{eqnarray}}
\def\eea{\end{eqnarray}}
\def\beas{\begin{eqnarray*}}
\def\eeas{\end{eqnarray*}}

\def\D{{D}}
\def\Z{{Z}}
\def\Y{{\bf Y}}

\def\V{{\bf V}}
\def\U{{\bf U}}
\def\Z{{\bf Z}}
\def\W{{\bf W}}
\def\P{{\bf P}}
\def\I{{\bf I}}
\def\D{{\bf D}}
\def\bx{{\boldsymbol \xi}}
\def\bet{{\boldsymbol\eta_0}}

\def\be{\begin{equation}}
\def\ee{\end{equation}}

\begin{document}

\title{Automatic computation of quantum-mechanical bound states and wavefunctions} 
\author{V. Ledoux\footnote{Postdoctoral
Fellow of the Research Foundation-Flanders (FWO)} and M. Van Daele}

\maketitle

\begin{abstract}
We discuss the automatic solution of the multichannel Schr\"odinger equation. The proposed approach is based on the use of a CP method for which the step size is not restricted by the oscillations in the solution. Moreover, this CP method turns out to form a natural scheme for the integration of the Riccati differential equation which arises when introducing the (inverse) logarithmic derivative.  A new Pr\"ufer type mechanism which derives all the required information from the propagation of the inverse of the log-derivative, is introduced. It improves and refines the eigenvalue shooting process and implies that the user may specify the required eigenvalue by its index. 
\end{abstract}



\section{Introduction}




The matrix-vector Schr\"odinger equation arises in various scattering and bound-state problems in physics and chemistry \cite{Alexander,Hutson,Ledoux3,Thompson}. 
It may arise for instance, as a result of applying the so-called coupled channel approach which separates the radial coordinate from the rest of the variables in a multidimensional Schr\"odinger equation (see \cite{Hutson,Thompson}). 
The resulting system of $n$ coupled Schr\"odinger equations may be written in matrix form as
\begin{equation}\label{eq1}
\frac{d^2 {\Y}}{dx^2}=\left[{\V}(x)-E{\I}\right]{\Y}(x),
\end{equation}
where ${\Y}$ is a vector function of order $n$, $\I$ is the $n\times n$ unit matrix, and ${\V}$ is the symmetric $n\times n$ potential matrix with elements $V_{ij}(x)$. 

We suppose that equation \eqref{eq1} is supplemented by appropriate (bound-state) boundary conditions in the endpoints of the integration interval $[a,b]$.
In the regular case, we write these boundary conditions in the following general form:
\begin{eqnarray}\label{bcs0}
\nonumber{\bf A_1}\Y(a)+{\bf A_2}{\Y'}(a)={\bf 0}\\
{\bf B_1}{\Y}(b)+{\bf B_2}{\Y'}(b)={\bf 0}
\end{eqnarray}
where ${\bf A_1}, {\bf A_2}, {\bf B_1}, {\bf B_2}$ are real $n$ by $n$ matrices satisfying the conjointness conditions
\begin{eqnarray}\label{conj}
\nonumber{\bf A_1}^{T}{\bf A_2}-{\bf A_2}^{T}{\bf A_1}={\bf 0}\\
{\bf B_1}^{T}{\bf B_2}-{\bf B_2}^{T}{\bf B_1}={\bf 0},
\end{eqnarray}
and the rank conditions $ \mbox{rank}({\bf A_{1}}|{\bf A_{2}}) = n$, $\mbox{rank}({\bf B_{1}}|{\bf B_{2}}) = n$.
Here $({\bf A_{1}}|{\bf A_{2}})$ denotes the $n\times 2n$ matrix whose first $n$ columns are the
columns of ${\bf A_{1}}$ and whose $(n+1)^{\rm st}$ to $2n^{\rm th}$ columns are the columns of ${\bf A_{2}}$.
The objective of a bound-state eigenvalue calculation is to locate energies $E_k$ for which there exist solutions of \eqref{eq1} that satisfy the bound-state boundary conditions.
Only for the scalar case $n = 1$, it is guaranteed that all the eigenvalues
are simple and distinct. For $n > 1$ however, any of the
eigenvalues may have a multiplicity as great as $n$.

Solving systems of coupled channel Schr\"odinger equations, poses some challenging computational problems: (i) the efficient propagation of the oscillatory solution on the integration domain, (ii) instabilities in the presence of closed channels ($V_{jj}>E$), (iii) the determination of the eigenvalues, their multiplicities and their indices, (iv) the construction of the normalized wave functions,.... 

The highly oscillatory behavior of the solutions corresponding to high eigenvalues forces a naive integrator
to take increasingly smaller steps. However, some classes of methods do exists which allow the computationally efficient propagation of the solution for higher $E$ values. In section \ref{section2}, we briefly revisit the CP methods, a class of methods which has been described in detail by Ixaru in \cite{Ixarubook} and still proofs to be very succesful for the numerical solution of the Schr\"odinger equation. For these methods, the step size is not restricted by the oscillations in the solution and the cost can be bounded independently of the eigenparameter $E$. 
Moreover a CP method is well suited to be used in a shooting approach, where the system is repeatedly integrated for different trial values of the eigenvalue. 

When using  an oscillation-proof CP-based method to solve multichannel Schr\"odinger equations in the presence of closed channels, one needs to address simultaneously problem (ii), which is a well-known difficulty in the theory of close-coupled equations \cite{Ershov,mano}, and is symptomatic of solution techniques which propagate the wavefunction explicitly. In section \ref{section3} we describe an invariant imbedding method, in which the propagated quantity is not the wavefunction matrix but its inverse logarithmic derivative. The idea of using invariant imbedding in the context of Schr\"odinger equations has found broad application, especially for large coupled-channel calculations, and dates back at least to \cite{Johnson2,Johnson,Light,Mrugala}. In the present paper, we will show that a CP method forms a natural scheme for the integration of the Riccati differential equations which arise when introducing the inverse logarithmic derivative. 


In section \ref{section4}, we consider the solution of the eigenvalue problem, which has attracted much less attention in literature than the scattering problem.  In order to improve and refine the eigenvalue shooting process, we supplement our procedure with a new Pr\"ufer transform type mechanism which derives all its required information from the propagation of the inverse logarithmic derivative. The Pr\"ufer algorithm allows to determine eigenvalue indices and multiplicities. This also implies that the user may specify the required eigenvalue by its index and does not need to provide any initial guess.


Section 5 describes the implementation of the different algorithms in an automatic software code in greater detail. In section 6 we look at
some results generated by this software package. 






\section{Constant Perturbation (CP) methods}\label{section2}
Methods which are succesful in the presence of high oscillation are often based on some form of coefficient approximation (see e.g.\ \cite{Iserles}). The coefficient functions in the differential equation are approximated and the corresponding solution is constructed analytically. In the context of the Schr\"odinger equation, coefficient approximation translates into the (piecewise) approximation of the potential function. Many references on approximate potential algorithms e.g.\ \cite{Canosa,Chou,Light,mano2,Marletta,Pruess}, confine themselves to piecewise constant approximation for which the approximating problem can be integrated explicitly in terms of trigonometric or hyperbolic functions. 

Let us review how this piecewise constant approximation method can be used to propagate the solution of the coupled channel  initial value problems which arise in an eigenvalue shooting process. The bound-state wavefunction $\Y(x)$ in eq.\ \eqref{eq1}, satisfying the boundary conditions in $a$ and $b$, is represented as a column vector with $n$ components. However, if the boundary conditions are applied at only one end of the range, there are $n$ linearly independent solution vectors that satisfy the Schr\"odinger equation, so that until {\em both} boundary conditions are applied it is actually necessary to propagate an $n\times n$ wavefunction matrix. So one can think of $\Y$ as an $n\times n$ matrix here. 

A partition of $[a,b]$ is introduced, with the mesh points $x_0=a,x_1,x_2,...,x_{N}=b$. In each interval $[x_{i},x_{i+1}], h_i=x_{i+1}-x_{i}$, the solution and its first derivative are advanced by a blockwise propagation algorithm:
\begin{equation}\label{eq0}\left[\begin{matrix}\Y(x_{i+1})\\
{\Y}'(x_{i+1})\end{matrix}\right]=
\left[\begin{matrix}\U(h_i)&\W(h_i)\\
\U'(h_i)&\W'(h_i)\end{matrix}\right]\left[\begin{matrix}\Y(x_{i})\\ \Y'(x_{i})\end{matrix}\right].
\end{equation} 
The elements of the transfer matrix, ${\U(\delta)}$ and ${\W(\delta)}$ are the $n \times n$ solutions of 
\begin{equation}\label{eq1b}{\P}''=\left({\V}(x_i+\delta)-E{\I}\right){\P},\;\;\delta \in [0,h_i]\end{equation}
corresponding to the initial conditions ${\P}(0)={\I}$, ${\P}'(0)={\bf 0}$ and ${\P}(0)={\bf 0}$, ${\P}'(0)={\I}$, respectively.
To determine ${\U}$ and ${\W}$ the potential matrix is approximated by a constant matrix,
\begin{equation}{\V}(x_i+\delta)\approx{\V}_0.\end{equation}
The symmetric matrix ${\V}_0$ is then diagonalized and let ${\D}$ be the orthogonal diagonalization matrix, i.e. ${\V}_0={\D}{\V^D_0} {\D}^T$. We can then write Eq. \eqref{eq1b} as
\begin{equation}{\D}^T{\P}''{\D}=\left({\D}^T {\V_0}{\D} -E{ \I}\right){ \D}^T{\P}{ \D},\end{equation}
or in the ${\D}$ representation
\begin{equation}{\P^D}''=\left({\V^D_0} -E{ \I}\right){\P^D}.\label{sets}\end{equation}
The diagonalization process transforms the system into one in which there is no coupling. The resulting set \eqref{sets} of $n$ one-dimensional Schr\"odinger equations can then be solved analytically. That is, \eqref{sets} is solved for ${\U^D}$ and ${\W^D}$; the initial conditions are the same as in the original representation. 
Let the functions $\xi(z)$ and $\eta_0(z)$ (as in \cite{Ixarubook}) be defined as follows:
\begin{equation}
\nonumber \xi(z) = \left\{ \begin{array}{ll}
\displaystyle \cos(|z|^{1/2}) & {\rm\ if\ } z \leq 0\,,\\[.2cm]
\displaystyle \cosh(z^{1/2}) & {\rm\ if\ } z > 0\,,
\end{array} \right.
\quad
\eta_0(z) = \left\{ \begin{array}{ll}
\displaystyle \sin(|z|^{1/2})/|z|^{1/2} & {\rm\ if\ } z < 0\,,\\[.2cm]
\displaystyle 1 & {\rm\ if\ } z=0\,,\\[.2cm]
\displaystyle \sinh(z^{1/2})/z^{1/2} & {\rm\ if\ } z > 0\,.
\end{array} \right.
\end{equation}
The propagators ${\U^D}(\delta)$ and ${\W^D}(\delta)$ are then diagonal matrices:
\begin{align}
{\U^D}&=({\W^D})'={\bx}({\Z(\delta)})\label{UWdef0}\\
\delta({\U^D})'&={\Z}(\delta)  \bet({\Z(\delta)})\\
{\W^D}&=\delta\bet({\Z(\delta)})\label{UWdef}\end{align}
where
\begin{equation}{\Z}(\delta)=(\V^D_{0}-E\I)\delta^2,\end{equation}
and ${\bx}({\Z}),\,\bet({\Z}) $  are diagonal matrices of functions with diagonal elements ${\xi}({Z_k}), {\bf\eta}_0({Z_k})$
with $Z_k(\delta)=(V^D_{0_{kk}}-E)\delta^2$.
 Once the values at $h_i$ of the ${\U^D,\, \W^D}$ matrices and of their derivatives have been evaluated, they are reconverted to the original representation to obtain the desired entries of the transfer matrix in \eqref{eq0}. 
The diagonalising transformation is energy independent. In a bound-state problem which involves calculations at a number of `trial' and `iterated' energies this is clearly a good feature.

\label{CA6}
The method based on a constant approximation of the potential, is only second order. To obtain higher-order coefficient approximation methods, piecewise
polynomial approximations of higher degree should be used \cite{Ixarubook}. In fact, it can be shown that piecewise polynomial interpolants  ${\bf V}^{(m)}$ of degree $m$ through the Legendre nodes, give rise to a method of order $2m+2$. 
It is, however, difficult to obtain analytic expressions for the solution of the approximating problem when $m>0$. Fortunaly, a higher order polynomial approximation still retaining the nice property of having explicit integration in terms of trigonometric/hyperbolic functions, can be realized using a perturbation approach by Ixaru \cite{Ixarubook, Ixaru, Ixaru2} leading to the so-called CP methods.
 The idea is to derive correction terms from the perturbation ${\bf V}^{(m)} (x_i+\delta)-\V_0$ and add these to the second order method \eqref{UWdef0}-\eqref{UWdef} to obtain more accurate approximations for the propagator matrices ${\bf U}$, ${\bf W}$ and their first derivatives. We refer to \cite{Ixaru,Ledoux2} for more details on the perturbative procedure and for the formulae of some higher order schemes.

\section{Invariant imbedding}\label{section3}
The CP schemes propagate the wavefunction matrix and its derivative explicitly. However, wavefunction propagation methods are subject to a classical numerical instability. In many bound-state calculations originating from quantum physics, all channels are locally closed ($V>E$) near the origin or at large distance, and it is often necessary to include channels that are closed at all $x$ to get converged results. The wavefunction component $Y_j$ for a locally open channel ($V_{jj}<E$) is an oscillatory function of $x$, but the $Y_j$ for a locally closed channel ($V_{jj}(x)>E$) is made up of exponentially increasing and decreasing components. 
If there are both locally open and locally closed channels over any range of $x$, there is a tendency for the closed channel components to grow so quickly that because of numerical rounding errors the linear independence of the different solutions in the wavefunction matrix is lost.

In e.g.\ \cite{Friedman,Gordon,Ixaru,Tolsma}, stabilizing transformations were applied during propagation to prevent overflow and to maintain the linear independence of the $n$ solution vectors. 
The drawback with such transformations is that they are expensive and should be applied across the entire integration range when strongly closed channels are included.
A more satisfactory solution is to form a CP propagator that is stable in the presence of closed channels. That is, we use the
knowledge of the components of the transfer matrix in \eqref{eq0} to construct a propagation algorithm  
for the inverse of the log-derivative ${\bf \Psi}={\Y}({\Y}')^{-1}$, sometimes called the $R$-matrix, see e.g.\ \cite{Hutson}:
\begin{equation}
{\bf \Psi}(x_{i+1})=[{\W}(h_i)+{\U}(h_i){\bf \Psi}(x_i)][{\W}'(h_i)+{\U}'(h_i){\bf \Psi}(x_i)]^{-1}.\label{Rrec}
\end{equation}
Apart from the fact that no stabilising transformation is needed, ${\bf \Psi}$  contains the minimum amount of information needed for the determination of the bound-state eigenvalues. Since ${\bf \Psi}$ is symmetric for all $x$, the $2n^2$ storage locations required in the $(\Y,\Y')$ representation can be reduced to $n(n+1)/2$ for calculations in the ${\bf \Psi}$  representation. 

%

There is however one complication in propagating ${\bf \Psi}$. The matrix ${\bf \Psi}={\Y}({\Y}')^{-1}$ has a singularity whenever the determinant of ${\Y}'$ vanishes.  
\begin{figure}
	\centering
	\includegraphics[width=0.85\textwidth]{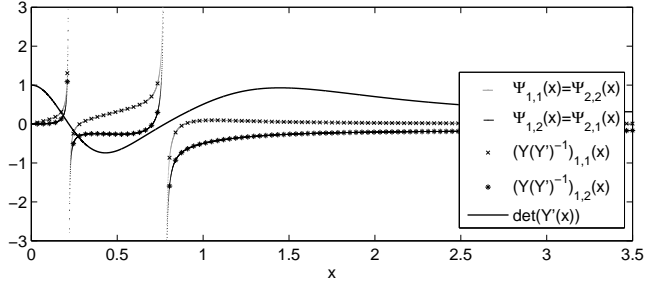}
	\caption{Illustration of how $\Psi(x)$ is integrated through the singularities by a CP method. }	\label{fig:Psi}
\end{figure}
By introducing ${\bf \Psi}$, the system of coupled equations is recast into a matrix Riccati differential equation (see \cite{Atkinson,Chou})
\begin{equation}
{\bf \Psi}'=\I-{\bf \Psi}^T[{\V}-E \I]{\bf \Psi}.\label{Psieq2}
\end{equation} 
One of the properties of Riccati type differential equations is indeed the existence of solutions with singularities. General numerical methods are incapable of integrating through these singularities and specific computational approaches should be used, see \cite{Chou,Schiff} and references therein.
It is known that one can integrate through singularities of the (inverse) log-derivative Riccati equation by changing coordinates and e.g.\ switch between the log-derivative and its inverse \cite{Nelson,Scottbook}. A new observation, however, is that when CP propagation algorithms of the form \eqref{Rrec} are used, changing coordinates is not required. Figure \ref{fig:Psi} shows the entries of ${\bf \Psi}$ for the testproblem described in \cite{IxaruE} which has known analytic expressions for the eigenvalues and eigenfunctions. 
Although, the ${\bf \Psi}$ curves have singularities when $\det(\Y'(x))$ is zero, they coincide with the entries of $\Y(\Y')^{-1}$ which were obtained by subdividing the known value of the wavefunction matrix at some $x$-values by its first derivative. The explanation of this good behaviour of the CP methods can be found in \cite{Schiff}. 
In \cite{Schiff}, Schiff and Shnider proposed so-called M\"obius schemes (since they use generalized M\"obius transformations) for the numerical integration of a Riccati differential equation. M\"obius schemes are natural schemes for the integration of Riccati differential equations from a geometric viewpoint and can accurately pass through singularities in the solution. The CP scheme \eqref{Rrec} can actually be seen as a M\"obius scheme. Indeed
the formula \eqref{Rrec} expresses the fact that ${\bf \Psi}(x_{i+1})$ is obtained from ${\bf \Psi}(x_i)$ by a M\"obius transformation which is independent of the value of  ${\bf \Psi}(x_i)$. 

\section{Shooting}\label{section4}
Our objective is to apply the CP integration schemes \eqref{Rrec} in a shooting procedure to locate the eigenvalues of the boundary value problem. We associate with the system the left and right-hand matrices ${\bf \Psi}_L$ and ${\bf \Psi}_R$. These are $n$ by $n$ matrix functions satisfying the initial conditions
\begin{align}
{\bf \Psi}_L(a)={\bf A_2}{\bf A_1}^{-1},\quad {\bf \Psi}_R(b)={\bf B_2}{\bf B_1}^{-1}.
\label{YLYRIC}
\end{align}
The matching condition can be expressed very simply in terms of ${\bf \Psi}$:
\begin{equation}{D}(E)=\det \left(\begin{matrix}{\bf \Psi}_L(c)-{\bf \Psi}_R(c)\end{matrix}\right)=0.\label{mc2}\end{equation}
where $c\in[a,b]$ is the matching point. 
So the basis for our numerical procedure is to integrate ${\bf \Psi}$ from the ends to $c$ for some trial value of $E$, evaluate ${D}(E)$ and take this as the mismatch. If the mismatch is not zero, $E$ is not an eigenvalue and the procedure is repeated for an adjusted value of $E$.

There are however some complications we need to deal with when we want to use this approach in automatic software. The mismatch function is an oscillating function in which singularities may appear, which is inconvenient for rootfinding. Moreover, the mismatch function does not change sign as $E$ passes through an eigenvalue of even multiplicity. This latter issue would be avoided when using the alternative approach (see \cite{Hutson}) of using the smallest eigenvalue in absolute value of the matching matrix ${\bf \Psi}_L(c)-{\bf \Psi}_R(c)$ as mismatch, in which case the number of zero eigenvalues of the matching matrix gives us the multiplicity of the eigenvalue. An additional problem, however, is that the mismatch function does not give any
way of determining the index of the eigenvalue once it has been found. 

In the scalar case, the problems are resolved by enhancing the algorithm with a Pr\"ufer method which counts the zeros of the solution as part of the integration for each trial $E$ value (see \cite{Bailey,Pruess,Pryce}), and which allows the computation of a particular eigenvalue of given index, without the prior computation of all the preceding eigenvalues. In the multichannel case, the index of an eigenvalue is determined by the number of times the determinant of the solution vanishes or more precisely by the number of times the wavefunction matrix has a zero eigenvalue and the multiplicity of this eigenvalue. We introduce a new Pr\"ufer-like procedure for the multichannel problem counting the number of zero eigenvalues and their multiplicities.
This procedure will allow us to construct an indexing function ${\mathcal I}(E)$, such that ${\mathcal I}(E)$ equals the number of eigenvalues that are less than $E$.
If we can calculate this function from shooting data, we are able to determine
whether a trial value of $E$ is ``near'' the eigenvalue $E_k$ we are looking for and whether it is too high or too low. 


Before proceeding to the multichannel case, it is instructive to consider first the Pr\"ufer approach in the single-channel Schr\"odinger case.
\subsection{The classical (scalar) Pr\"ufer method}
In a shooting process, the one-dimensional equation
\begin{equation}\label{1dim}
y''(x)=\left[V(x)-E\right]y(x),\;\;x\in[a,b],
\end{equation} is integrated from left to right, with initial values $y(a)=A_2, y'(a)=-A_1$, to obtain a left solution $y_L, y'_L$; and integrated from right to left, with initial values $y(b)=B_2, y'(b)=-B_1$, to obtain a right solution $y_R, y'_R$. The main idea of the Pr\"ufer method is to introduce polar coordinates $(\rho,\theta)$ in the phase plane:
\begin{equation}\label{uw1}
y=\rho \sin\theta,\quad y'=\rho \cos\theta.
\end{equation}
The phase angle $\theta$ is defined (modulo $\pi$) by the equation
\begin{equation}
\tan\theta=\frac{y}{y'}.\label{uw}
\end{equation}
Using \eqref{1dim} and \eqref{uw1} it can be shown that $\theta=\theta(x)$ satisfies (see \cite{Pryce}):
\begin{equation}\label{thetaeq}
\theta'=\cos^2 \theta-[V(x)-E] \sin^2 \theta,\quad a<x<b.
\end{equation}
Equation \eqref{thetaeq} has a left-solution $\theta_L(x)$, with $\theta_L(a)=\theta_0(a)$; and a right-solution $\theta_R(x)$, with $\theta_R(b)=\theta_0(b)$, where 
 $\theta_0(a)$ and $\theta_0(b)$ are defined by
\begin{align}\label{normal}
\tan \theta_0(a)=\left(-\frac{A_2}{A_1}\right),\quad 0\leq\theta_0(a)<\pi,\\
\tan \theta_0(b)=\left(-\frac{B_2}{B_1}\right),\quad 0<\theta_0(b)\leq\pi.
\end{align}
From equation \eqref{thetaeq} we see that if $\theta({\bar x})=m\pi$ (where $m$ is an integer), then $\theta'({\bar x})=1>0$. This shows that $\theta_L(x)$ increases through multiples of $\pi$ as $x$ increases.
Similarly $\theta_R(x)$ decreases through multiples of $\pi$ as $x$ decreases. Since $y=0$ just when $\theta$ is a multiple of $\pi$, the number of zeros of $y$ on $(a,c)$ is then the number of multiples of $\pi$ (strictly) between $\theta_L(a)$ and $\theta_L(c)$. Analogously the number of zeros of $y$ on $(c,b)$ is the number of multiples of $\pi$  through which $\theta_R$ decreases going from $b$ to $c$. Since the index $k$ of an eigenvalue equals the number of zeros of the associated eigenfunction $y(x)$ on the open interval $(a,b)$, we can use these results to formulate the function ${\mathcal I}(E)$. Let
\begin{equation}\label{intn}
\theta_L(c,E)-\theta_R(c,E)=n(c,E)\pi+\omega(c,E)
\end{equation}
where $n(c,E)$ is an integer and $0\leq \omega(c,E)<\pi$.
${\mathcal I}(E)$ can then be defined as ${\mathcal I}(E)=n(c,E)+1$
or
\begin{equation}
{\mathcal I}(E)=\frac{1}{\pi}\left[\theta_L(c,E)-\theta_R(c,E)-\omega(c,E)\right]+1.
\end{equation}
The function ${\mathcal I}(E)$ is a piecewise constant with jumps at the eigenvalues.

In practical implementations, one often uses a scaled version of the Pr\"ufer method 
\begin{equation}
y=S^{-1/2}\rho \sin\theta,\quad y'=S^{1/2}\rho \cos\theta,
\end{equation}
where $S$ is a scaling function chosen to give good numerical behaviour (see \cite{Pryce}). For instance, the software package SLEDGE \cite{Pruess}, which implements the second order coefficient approximation method for the scalar (Sturm-Liouville) problem, uses the fact that
\be\arctan \frac{\tau_i y_{i+1}}{y'_{i+1}} =\arctan \frac{\tau_i y_i}{y'_i}+\tau_ih_i \mod \pi,\quad \tau_i=\sqrt{E-V_0}\label{sledge1}\ee
when $E-V_0>0$, to determine the number of zeros in the interval $(x_i,x_{i+1})$  as the number of integers 
in the interval $(\Theta/\pi,(\Theta+\tau_i h_i)/\pi)$ where $\Theta=\arctan \frac{\tau_i y_i}{y'_i}\label{sledge2}$. When $E-V_0<0$, the zero count is kept by noting that $y$ has a (single) zero in $(x_i,x_{i+1})$, if and only if $y_iy_{i+1}<0$.


\subsection{Indexing in the matrix case} \label{pruferm}
We try to generalize the Pr\"ufer idea to a system of equations by decoupling the system into $n$ scalar ones to which the simple Pr\"ufer method can be applied. This means that we try to obtain a problem in
diagonal form. In order to derive a direct generalization of the scalar Pr\"ufer approach, we will work here with the matrix function ${\bf \Psi}(x)=\Y(x){\Y'(x)}^{-1}$, whose scalar equivalent appears in \eqref{uw}. 
The matrices ${\bf \Psi}_L(x)$ and ${\bf \Psi}_R(x)$ are symmetric for all $x\in[a,b]$ (see \cite{Atkinson}).
The matrix ${\bf \Psi}$ can consequently be diagonalized to a matrix having the eigenvalues of ${\bf \Psi}$ on the diagonal.
These diagonal elements are then, as for the scalar case, used to devise a monotone increasing integer-valued indexing function.  Our approach owes a lot to the work of Atkinson \cite{Atkinson} who developed a Pr\"ufer method for the vector Sturm-Liouville system, and to Marletta \cite{Marletta,Marlettaphd} who sharpened Atkinson's method by finding an integer-valued function which jumps at each eigenvalue.
However in contrast to Atkinson and Marletta, we will not introduce the matrix ${\bf \Theta}(x)=(\Y'+i\Y)(\Y'-i\Y)^{-1}$. Instead we will show that sufficient information can be derived for our purpose here from the propagation of ${\bf \Psi}(x)$. Moreover, by introducing a SLEDGE-like node count algorithm, the new approach is more robust.
%

The matching condition \eqref{mc2} gives us the following lemma, which also forms the basis of the eigenvalue determination process described in \cite{Hutson}:
\begin{lemma}\label{lemma1}
$E$ is an eigenvalue of \eqref{eq1} when $\det({\bf \Psi}_L(c)-{\bf \Psi}_R(c))=0$. This is equivalent to zero being an eigenvalue of ${\bf \Psi}_L(c)-{\bf \Psi}_R(c)$. Moreover the multiplicity of 0 as an eigenvalue of ${\bf \Psi}_L(c)-{\bf \Psi}_R(c)$ is equal to the multiplicity of $E$ as an eigenvalue of \eqref{eq1}.
\end{lemma}



Let us denote the eigenvalues of ${\bf \Psi}_L$ and ${\bf \Psi}_R$  by $\{\gamma_j^L|j=1,\dots,n\}$ and $\{\gamma_j^R|j=1,\dots,n\}$, respectively. Analogously to the scalar case, the {\em phase angles} $\phi_j^L,\phi_j^R$ are then defined (modulo $\pi$) by
\[\tan \phi_j^L=\gamma_j^L,\tan \phi_j^R=\gamma_j^R\]
and uniquely determined functions when normalized by the conditions
\begin{align}
\phi_1^L\leq\phi_2^L\leq\dots\leq\phi_n^L\leq\phi_1^L+\pi\\
\phi_1^R\leq\phi_2^R\leq\dots\leq\phi_n^R\leq\phi_1^R+\pi\\
0\leq\phi_j^L(a)<\pi, 0<\phi_j^R(b)\leq\pi.
\end{align}
Choose $c\in[a,b]$ and let the eigenvalues of $[{\bf \Psi}_L(c)-{\bf \Psi}_R(c)][\I+{\bf \Psi}_L(c){\bf \Psi}_R(c)]^{-1}$ be $\tan \omega_j$
where the $\omega_j$ are normalized by the condition
\begin{equation}0\leq\omega_j<\pi.\label{normali}\end{equation}
We can now formulate the indexing function for the matrix case.

\begin{theorem} The function ${\mathcal I}(E)$ defined as
\[{\mathcal I}(E)=\frac{1}{\pi}\left(\sum_{j=1}^n \phi_j^L(c)-\sum_{j=1}^n \phi_j^R(c)-\sum_{j=1}^n \omega_j(c)\right)+n\]
 is an integer valued increasing step function of $E$ whose points of increase are the eigenvalues of the problem \eqref{eq1}. When $E$ increases through an eigenvalue of multiplicity $m$, ${\mathcal I}$ increases by $m$.
\end{theorem}
\begin{proof} Let us use the notations
\begin{equation}
S({\bf \Psi}_L(x))=\sum_{j=1}^n \phi_j^L(c), \quad
S({\bf \Psi}_R(x))=\sum_{j=1}^n \phi_j^R(c).
\end{equation}
By definition of the $\omega_j$,
\[\sum_{j=1}^n \omega_j(c)=S\left([{\bf \Psi}_L-{\bf \Psi}_R][\I+{\bf \Psi}_L{\bf \Psi}_R]^{-1}\right) \mod \pi.\]
Also, using the arctangent addition formula,
\[S\left([{\bf \Psi}_L-{\bf \Psi}_R][\I+{\bf \Psi}_L{\bf \Psi}_R]^{-1}\right) =(S({\bf \Psi}_L)-S({\bf \Psi}_R)) \mod \pi.\]
Combining these we have
\begin{equation}\sum_{j=1}^n \omega_j(c)=(\sum_{j=1}^n \phi_j^L(c)-\sum_{j=1}^n \phi_j^R(c)) \mod \pi,\label{wja}\end{equation}
which shows that ${\mathcal I}$ is integer valued.

Since each $\phi_j^L$ and $\phi_j^R$ is a continuous function of $E$, the only way that ${\mathcal I}$ can change value is if a $\omega_j$ passes through a multiple of $\pi$ and jumps because of the normalization \eqref{normali}. When this happens, 0 is an eigenvalue of ${\bf \Psi}_L(c)-{\bf \Psi}_R(c)$ or equivalently (by Lemma \ref{lemma1}) $E$ is an eigenvalue of \eqref{eq1}.

It remains to show that ${\mathcal I}$ increases by $m$ when $E$ passes through an eigenvalue of multiplicity $m$. Having Lemma \ref{lemma1}, it is sufficient to show that each $\omega_j$ can only {\em increase} through $\pi$ with increasing $E$, thus jumping down to $0$ again and causing ${\mathcal I}$ to increase. 
Theorem 10.2.3 in \cite{Atkinson} says that
\begin{equation}\frac{\partial {\bf \Psi}_L(c;E)}{\partial E}=[({\Y'_L}^{T})(c;E)]^{-1}\int_a^c \Y_L^T(t;E)\Y_L(t,E)dt\,[{\Y'_L}(c,E)]^{-1},\label{fr1}\end{equation}
and
\begin{equation}\frac{\partial {\bf \Psi}_R(c;E)}{\partial E}=-[({\Y'_R}^{T})(c;E)]^{-1}\int_c^b \Y_R^T(t;E)\Y_R(t,E)dt\,[{\Y'_R}(c,E)]^{-1}.\label{fr2}\end{equation}
The right hand side of \eqref{fr1} is positive definite and the right hand side of \eqref{fr2} is negative definite.  From Theorem V.2.3 in \cite{Atkinson}, we know that this means that the eigenvalues of ${\bf \Psi}_L(c)$ are increasing functions of $E$ and that the eigenvalues of ${\bf \Psi}_R(c)$ are decreasing functions of $E$. From \eqref{wja} it then follows that each $\omega_j$ increases through $\pi$ with increasing $E$.
\end{proof}
The value of ${\mathcal I}(E)-n$ for $E<E_0$ is equal to $-n$, independent of the potential matrix ${V}$ and the boundary conditions. This is clear for the trivial case of a diagonal potential matrix, where the system reduces to a superposition of $n$ scalar problems. For general potential matrices, this result can be derived from the close connection between the ${\mathcal I}(E)$ function defined here and the indexing function ${M}(E)$ proposed by Marletta in \cite{Marletta,Marlettaphd}.

 As their scalar equivalents $\theta$, the phase angles can never decrease through a multiple of $\pi$:
\begin{lemma}\label{lemman}
The `phase angles' $\phi_j^L$ and $\phi_j^R$ can only increase through multiples of $\pi$ with increasing $x$.
\end{lemma}
\begin{proof}
This is equivalent to showing that the derivative of the matrix function ${\bf \Psi}$ is positive definite when applied to eigenvectors of ${\bf \Psi}(x)$ associated with a zero eigenvalue (see \cite{Atkinson}). Supposing that for some $x$ one or more of the $\gamma_j(x)$ is zero, we consider any column matrix $w$ such that
\[{\bf \Psi}(x)w={\bf 0}\,\quad w^Tw>0.\]
Then
\[w^T {\bf \Psi}' w=w^T (I-{\bf \Psi}^T(\V-E \I){\bf \Psi})w=w^Tw >0.\]
This proves that the eigenvalues $\gamma_j(x)$ can only increase through 0. Thus the functions $\phi_j(x)$ are strictly increasing when they are multiples of $\pi$. 
\end{proof}

How can we now compute ${\mathcal I}(E)$ numerically? We need to integrate ${\bf \Psi}_L$ from $x=a$ to $x=c$, and ${\bf \Psi}_R$ from $x=b$ to $x=c$, while following the quantities $S({\bf \Psi}_L(x))$, $S({\bf \Psi}_R(x))$ continuously
and count the number of multiples of $\pi$ in each. 
We assume that the matrix function $\V=\V_0$ is piecewise constant and use the second order coefficient approximation method for the integration of the ${\bf \Psi}$ matrix, be it ${\bf \Psi}_{L}$ or ${\bf \Psi}_{R}$, across a mesh interval $[x_i,x_f]$ in the case ${\bf \Psi}_{L}$ or $[x_f,x_i]$ in the case ${\bf \Psi}_{R}$, with initial values known at $x_i$. 
As a first step, we consider again the diagonalization process. Since
\[{ \D}^{-1}{\bf \Psi}{ \D}={ \D}^{-1}\Y{ \D}{ \D}^{-1}{(\Y')}^{-1}{ \D}=\Y^{ D}{{(\Y'^{ D})}}^{-1}={\bf \Psi}^{ D}\]
we know that the eigenvalues of ${\bf \Psi}$ are precisely the same as those of ${\bf \Psi}^{ D}$. So we may forget about ${\bf \Psi}$ and think in terms of ${\bf \Psi}^{ D}$, which is given by 
\[{\bf \Psi}^{\D}(x)=\left[\delta{\bet}+{\bx}{\bf \Psi}^\D(x_i)\right]\left[\bx+(\Z/\delta){\bet}{\bf \Psi}^\D(x_i)\right]^{-1},\]
where $\delta=x-x_i$, $\bx = \bx(\Z(\delta)),\;\bet=\bet(\Z(\delta))$ and $\Z(\delta)=(\V_0^D-E \I)\delta^2$.

In order to compute $S({\bf \Psi}^\D)$ correctly, and not just modulo $\pi$, we need to know the number of times that some angle $\phi_j$ passes through a multiple of $\pi$ as $x$ moves from $x_i$ to $x_{f}$. To help us in this process we introduce the matrix ${\bf \Omega}$
\begin{align}
{\bf \Omega}(x)&=\left[\bx-\delta\bet{\bf \Psi}^\D(x_i)\right]^{-1}\left[\bx\Psi^\D(x_i)+\delta\bet\right]\\
&=\left[\I-\delta\bx^{-1}\bet{\bf \Psi}^\D(x_i)\right]^{-1}\left[{\bf \Psi}^\D(x_i)+\delta\bx^{-1}\bet\right]
\end{align}

\begin{lemma}\label{lemma22}
0 is an eigenvalue of ${\bf \Omega}$ if and only if 0 is an eigenvalue of ${\bf \Psi}^D$.
\end{lemma}
\begin{proof}
0 is an eigenvalue of ${\bf \Psi}^\D$ if and only if 0 is an eigenvalue of 
${\bx}^{-1}{\bf \Psi}^\D=(\delta{\bx}^{-1}{\bet}+{\bf \Psi}^\D(x_i))(\bx+(\Z/\delta){\bet}{\bf \Psi}^\D(x_i))^{-1}$ and of its transpose
${\bf \Psi}^\D{\bx}^{-1}=(\bx+{\bf \Psi}^\D(x_i)(\Z/\delta){\bet})^{-1}(\delta{\bx}^{-1}{\bet}+{\bf \Psi}^\D(x_i))$. 0 is an eigenvalue of ${\bf \Psi}^\D{\bx}^{-1}$ if and only if 0 is an eigenvalue of $(\delta{\bx}^{-1}{\bet}+{\bf \Psi}^\D(x_i))$ and consequently also of ${\bf \Omega}$.
\end{proof}

We will use the matrix ${\bf \Omega}$ in order to compute $S({\bf \Psi}^\D(x))$ from $S({\bf \Psi}^\D(x_i))$. 
Let the eigenvalues of ${\bf \Psi}^\D$ be $\tan \phi_j$ and those of ${\bf \Omega}$ be $\tan \varphi_j$, and suppose that,
\[\phi_j=\pi n_j+\rho_j,\quad \varphi_j=\pi m_j+\tau_j,\]
where $n_j$ and $m_j$ are integers and $\rho_j$ and $\tau_j$ lie in $[0,\pi).$ Then we can write
\be S({\bf \Psi}^\D)=S({\bf \Omega})+\sum_{j=1}^n (\rho_j-\tau_j)+\pi\sum_{j=1}^n(n_j-m_j).\label{eqP}\ee
The $\tau_j$ and $\rho_j$ are easily computed directly from ${\bf \Omega}$ and ${\bf \Psi}^\D$, because the number of multiples of $\pi$  is unambiguous.
The quantity $S({\bf \Omega})$ can be computed correctly using the identity 
\be S({\bf \Omega})=S({\bf \Psi}^\D(x_i))+S(\delta\bx^{-1}\bet)\label{SOm}\ee
which is based on the arctangent addition formula.
The term $S(\delta\bx^{-1}\bet)$ in \eqref{SOm} is computed by applying a Pr\"ufer transformation to each diagonal term in turn. 
Note that $\delta\bx^{-1}\bet$ equals the $R$-matrix ${\bf \Psi}^\D(x_i+\delta)$ one obtains after one constant coefficient approximation propagation step when starting from the initial values $\Y(x_i)={\bf 0}, \Y'(x_i)={\I}$ and thus ${\bf \Psi}^\D(x_i)={\bf 0}$. 
Therefore, for each $j$ let $y_j$ be the solution of the initial value problem
\[-y_j''+d_jy_j=0,\quad y_j(x_i)=0,\quad y'_j(x_i)=1,\]
where the $d_j$ are the elements of the diagonal matrix $E \I-{\V}_0^\D$.
The Pr\"ufer transformation 
\[y_j=\rho_j\sin\theta_j,\quad y_j'=\rho_j\cos\theta_j\]
is applied. Then $\theta_j$ satisfies the initial value problem
\[\theta_j'=\cos^2\theta_j+d_j\sin^2\theta_j,\quad \theta_j(x_i)=0\]
which means that
\begin{equation}\tan \theta_j(x_i+\delta)= \delta\eta_0(-d_j\delta^2)/\xi(-d_j\delta^2)\label{theteq}\end{equation}
Suppose that 
\[\theta_j(x_i+\delta)=s_j\pi+\kappa_j\]
with $\kappa_j$ in $[0,\pi)$.
Then we may compute
\[S(\delta\bx^{-1}\bet)=\sum_{j=1}^n\theta_j=\sum_{j=1}^n \kappa_j +\sum_{j=1}^n s_j\pi\]
The value $\kappa_j$ is given by \eqref{theteq} and $\sum_{j=1}^n s_j$ can be obtained by applying the same procedure as described at the end of the previous section using \eqref{sledge1}. That is: $s_j$ equals the number of integers in the interval $(0,\sqrt{d_j}h_i/\pi)$ when $d_j>0$. When $d_j\leq 0$, $s_j$ equals zero. In this way no multiples of $\pi$ in $\theta_j$ are missed. This is an improvement with respect to the procedure used in \cite{Marletta} and \cite{Ledoux}, where each eigenvalue of the wavefunction matrix is supposed to pass through zero not more than once when integrating over a mesh interval.

The only remaining unknown quantity in \eqref{eqP} is $\sum_{j=1}^n(n_j-m_j)$. 
When $x=x_i$, this sum vanishes. We will show that the same is true for all $x$. We know that a $n_j$ can only change when the corresponding $\phi_j$ passes through a multiple of $\pi$, and when this happens then there will be some $\varphi_k$ passing through a multiple of $\pi$ as well (by Lemma \ref{lemma22}), so $m_k$ will also change. In fact, there will be just as many $m_k$'s changing as there are $n_j$'s changing. From lemma \ref{lemman} we can deduce that $n_j$ can only increase. 
It only remains to show that also $m_k$ can only increase.
Let us define ${\hat x}\in[x_i,x_{f}]$ as the $x$ position where $\phi_j$ and $\varphi_k$ pass through a multiple of $\pi$.
To show that the angle $\varphi_k$ is increasing in ${\hat x}$, we define first the matrix ${\bf \Upsilon}$:
\[{\bf \Upsilon}(x)={\bf \Psi}^\D(x_i)+{\bf F}({\delta}),\quad {\rm with}\quad {\bf F}({\delta})={\delta}\bx^{-1}(\Z({\delta}))\bet(\Z({\delta})) \] 
which has a zero eigenvalue if and only if ${\bf \Omega}$ (or ${\bf \Psi}^D$) has a zero eigenvalue. Since ${\bf \Upsilon}'(x)=\I-(\V_0^\D-E\I){\bf F}({\delta})^2$ is a positive definite diagonal matrix, the eigenvalues of ${\bf \Upsilon}$ are increasing through zero.
Now suppose that as $x$ increases through ${\hat x}$, some `angle' of ${\bf \Upsilon}$ called $\vartheta_i$ increases through a multiple of $\pi$. Let $\varphi_k$ be the angle which is closer to $\vartheta_i$ (modulo $\pi$) than any other $\varphi_j$, for all $x$ close to ${\hat x}$. If $\varphi_k$ does not increase through a multiple of $\pi$ as $x$ increases through ${\hat x}$, then for all sufficiently large $x<{\hat x}$, we will have
\[\varphi_k(x)\geq m\pi,\quad\vartheta_i(x)<l\pi,\]
where $\varphi_k({\hat x})=m\pi$, and
$\phi_i({\hat x})=\vartheta_i({\hat x})=l\pi$.
Let us now define
\begin{align}\nonumber{\bar {\bf \Omega}}(x;\epsilon)&=\left[\I-{\bf F}({\delta}){\bf \Psi}^\D(x_i)\right]^{-1}\left[{\bf \Psi}^\D(x_i)+{\bf F}({\delta})+\epsilon (\I+{\bf F}({\delta})^2)\right], \\\nonumber
{\bar {\bf \Upsilon}}(x;\epsilon)&={\bf \Psi}^\D(x_i)+{\bf F}({\delta})+\epsilon  (\I+{\bf F}({\delta})^2)\end{align}
where $\epsilon$ is a small positive parameter close to 0. The derivative
\[\frac{\partial {\bar {\bf \Upsilon}}(x;\epsilon)}{\partial \epsilon}=\I+{\bf F}^2\]
is positive definite, and consequently ${\bar \vartheta}_i$ is an increasing function of $\epsilon$. 
This means that in the neighbourhood of ${\hat x}$ the angle ${\bar \vartheta}_i$ of ${\bar {\bf \Upsilon}}(x;\epsilon)$ will be just a bit larger than the angle $\vartheta_i$ of ${\bf \Upsilon}(x)={\bar {\bf \Upsilon}}(x;0)$. 
If we define the column vector $w$ such that ${\bf \Omega}({\hat x}) w = {\bf 0}$ or equivalently ${\bf \Psi}^D(x_i) w = -{\bf F}({\hat \delta}) w$. Then
\begin{align}
\nonumber w^T \frac{\partial {\bar {\bf \Omega}}({\hat x};\epsilon)}{\partial \epsilon} w &= w^T \left[\I-{\bf F}({\hat \delta}){\bf \Psi}^D(x_i)\right]^{-1} \left[\I+{\bf F}({\hat \delta})^2\right]w\\\nonumber
&=w^T \left[\I-{\bf F}({\hat \delta}){\bf \Psi}^D(x_i)\right]^{-1} \left[\I-{\bf F}({\hat \delta}){\bf \Psi}^D(x_i)\right]w=w^Tw>0
\end{align}
and the (zero) eigenvalue of ${\bar {\bf \Omega}}(x;\epsilon)$ increases with $\epsilon$ in ${\hat x}$. As a consequence the angle ${\bar \varphi}_k$ of ${\bar {\bf \Omega}}(x;\epsilon)$ will be larger than $\varphi_k$ of ${\bf \Omega}(x)={\bar {\bf \Omega}}(x;0)$ in a small neighbourhood of ${\hat x}$. 
For an arbitrarily small positive value of $\epsilon$ there exists thus an $x$-value close to ${\hat x}$ such that
\[{\bar \varphi}_k(x)>m\pi,\quad {\bar \vartheta}_i(x)=l\pi.\]
But this contradicts the fact that ${\bar {\bf \Upsilon}}(x;\epsilon)$ has a zero eigenvalue whenever ${\bar {\bf \Omega}}(x;\epsilon)$ has a zero eigenvalue. Thus $\varphi_k$ can only increase through multiples of $\pi$ and $m_k$ increases by one whenever $n_j$ increases by one.

We have then finally the following result
\begin{theorem}
Let the eigenvalues of ${\bf \Psi}^\D$ be represented as $tan(\rho_j)$ and the eigenvalues of ${\bf \Omega}$ as $tan(\tau_j)$, for $j=1,\dots,n,$ where
\be 0\leq\rho_j<\pi,\quad 0\leq\tau_j<\pi\ee
Then 
\[S({\bf \Psi}^\D(x))=S({\bf \Psi}^\D(x_i))+S(\delta\bx^{-1}\bet)+\sum_{j=1}^n (\rho_j-\tau_j).\]
\end{theorem}
And since $S({\bf \Psi}(x))=S({\bf \Psi}^\D(x))$, this result allows us to count the number of multiples of $\pi$ in $S({\bf \Psi}(x))$ as we integrate across an interval.


\section{The automatic solution of coupled-channel Schr\"odinger systems}
We discuss here the procedures as they were implemented in an Matlab software package for the automatic solution of the Schr\"odinger equation. The package is available at \verb|http://www.nummath.ugent.be/SLsoftware|.
\subsection{Eigenvalue computations}
Algorithm \ref{algo1} shows the shooting procedure in which a CP method is used to propagate the left-hand and right-hand ${\bf \Psi}$ matrices. In our implementations we used a sixth order CP method for this propagation.  A Newton-Raphson iteration procedure can be used  to obtain a new and better trial value for the eigenvalue since
 the CP algorithms allow a direct evaluation of the first derivative of $\U,\W$ and $\U',\W'$, and consequently of ${\bf \Psi}$, with respect to $E$, see \cite{Ixaru2,Ixaru}. 
\begin{algorithm}[t]
\caption{Shooting for the eigenvalue $E_k$}\label{algo1}
\begin{algorithmic}[1]
\STATE Input: a trial value $E$, a mesh $a=x_0<x_1<\dots<x_{N}=b$, with stepsizes $h_i=x_{i+1}-x_{i}$.
\STATE Choose a meshpoint $c=x_m, 0\leq m\leq N$ as the matching point.
\STATE Set up initial values for ${\bf \Psi}_L$ and ${\bf \Psi}_R$ satisfying the BCs at $a$ and $b$ resp.
\REPEAT 
\FOR{$i=0$ to $m-1$} 
\STATE Compute $\U(h_i),\W(h_i),\U'(h_i),\W'(h_i)$ by a CP method. 
\STATE Propagate ${\bf \Psi}_L$ over the interval $[x_i,x_{i+1}]$:\\
\hspace*{3mm}${\bf \Psi}_L(x_{i+1})=[\W(h_i)+\U(h_i){\bf \Psi}_L(x_i)][\W'(h_i)+\U'(h_i){\bf \Psi}_L(x_{i})]^{-1}$
\ENDFOR
\FOR{$i=N$ down to $m+1$} 
\STATE Compute $\U(h_i),\W(h_i),\U'(h_i),\W'(h_i)$ by a CP method.
\STATE Propagate ${\bf \Psi}_R$ over the interval $[x_{i-1},x_{i}]$:\\
\hspace*{3mm}${\bf \Psi}_R(x_{i-1})=[-\W(h_i)+\W'(h_i){\bf \Psi}_R(x_i)][\U(h_i)-\U'(h_i){\bf \Psi}_R(x_{i})]^{-1}
$
\ENDFOR
\STATE Adjust $E$ to solve the equation $\det({\bf \Psi}_L(c)-{\bf \Psi}_R(c))=0$.
\UNTIL{$E$ sufficiently accurate (e.g.\ until the difference between two subsequent $E$ values is smaller than some user input tolerance)} 
\STATE Compute the multiplicity of $E$ as the number of zero eigenvalues in ${\bf \Psi}_L(c)-{\bf \Psi}_R(c)$.
\end{algorithmic}
\end{algorithm}
A good initial guess is needed as input for the Newton-Raphson procedure, this is where the Pr\"ufer procedure from section \ref{pruferm} comes in. We use the Pr\"ufer procedure to generate reasonably tight upper and lower bounds for the eigenvalue $E_k$ sought, i.e.\ a bracket $[{\hat E},{\bar E}]$ is determined such that ${\mathcal I}({\hat E})= k$ and ${\mathcal I}({\bar E})= k+1$ or ${\bar E}-{\hat E}<tol$. $E=({\hat E}+{\bar E})/2$ is then passed to Algorithm \ref{algo1} as the initial trial value for the eigenvalue $E_k$. From the Pr\"ufer based indexing function ${\mathcal I}({\hat E})$, we can also deduce the multiplicity of the eigenvalue. As indicated in Algorithm 1, this multiplicity can, however, also be obtained by examining the number of zero eigenvalues in the matching matrix.

 Our invariant imbedding approach allows matching anywhere in the integration domain. However a well-considered choice for the matching position can reduce the work required to locate an eigenvalue when it delivers us well-behaved forms of the mismatch function and consequently faster convergence of the Newton iteration. In our implementation, we chose as matching point a meshpoint in the region of the potential minimum \cite{Dunker},i.e.\ $c$ is the right endpoint of the mesh interval where $\V_0^\D$ reaches its smallest value. 
Extra care (and a different matching point) may be needed when for certain $E$ values singularities appear close to the matching point and interfere with the Newton-Raphson process. This situation can be detected using a second mesh, which we call the reference mesh and which is also used for error control (see further). The finer reference mesh with stepsizes $h_1/4, h_1/4,(h_1+h_2)/4,(h_1+h_2)/4,(h_2+h_3)/4,(h_2+h_3)/4,\dots,h_N/4$ is constructed from the original mesh which has stepsizes $h_1,h_2,h_3,...,h_{N}$. In this way the reference mesh has (approximately) twice as many meshpoints and its meshpoints do in general not coincide with the ones from the original mesh.

\begin{remark}
The numerical Pr\"ufer procedure introduced in section \ref{section4} to compute ${\mathcal I}({ E})$ assumes that the matrix function ${\bf V}$ is piecewise constant. The numerical computation of ${\mathcal I}({ E})$ returns us thus the number of eigenvalues smaller than $E$ of the problem where the potential ${\bf V}$ has been replaced by the constant matrix ${\bf V}_0$ over each mesh interval. Note that this ``second order'' Pr\"ufer procedure is only used to obtain a first crude approximation for an eigenvalue and that for this purpose it is in many cases sufficient to apply the Pr\"ufer process over a ``coarse'' mesh, e.g.\ a mesh corresponding to the higher order CP method. Difficulties are only expected when one deals with a problem with close eigenvalues. By making use of the reference mesh, these difficulties can be detected. When the results obtained on the finer reference mesh do not agree (within some input tolerance $tol$) with the results on the original mesh, computations need to be repeated on a finer mesh.
\end{remark}

\subsection{Eigenfunctions}
We assume that the eigenvalue has already been found to a high level of accuracy, and we attempt now to find an associated eigenfunction. 
Suppose that we have available the left-hand solution matrices $\Y_L$ and $\Y'_L$ and the right-hand solution matrices $\Y_R$ and $\Y'_R$.
Any eigenfunction may be represented as
\begin{equation}y_k(x)=\begin{cases}\Y_L(x,E)w_L \quad (a\leq x <c)\\\Y_R(x,E)w_R \quad(c\leq x \leq b)\end{cases}\label{eigfu}\end{equation}
where $w_L$ and $w_R$ are appropriate non-zero vectors. In fact, 
$w_L$ is any non-trivial vector such that
\[({\Y'_R}^T\Y_L-\Y_R^T\Y'_L)(c,E)w_L={\bf 0},\]
and $w_R$ has to satisfy 
\[w_R^T({\Y'_R}^T\Y_L-\Y_R^T\Y'_L)(c,E)={\bf 0}^T.\]
For a simple eigenvalue, there is just one $w_L$ (up to scaling), and a corresponding $w_R$. 
If the eigenvalue in question has, however, multiplicity $m \leq n$ then one can find $m$ linearly independent vectors which can fulfill the role of $w_L$, and correspondingly, $m$ different $w_R$. Moreover any linear combination of valid $w_L$'s yields another valid $w_L$, and similar for $w_R$. We have to be careful to match the left-hand half of an eigenfunction with the right hand half of the same eigenfunction and to obtain the correct normalisation.  We used the procedure described in \cite{Marletta} for this purpose. 
This procedure assumes, however, the knowledge of the matrices $\Y_L$ and $\Y_R$. In our approach, not the matrices $\Y$ and $\Y'$ are propagated but ${\bf \Psi}$. The invariant imbedding variable ${\bf \Psi}$ is not sufficient to recover the eigenfunction values. We also need to have $\Y'$ in order to obtain the wavefunction matrix as $\Y={\bf \Psi} \Y'$. The matrix $\Y'$ is propagated by the following relation:
\[\Y'(x_{i+1})=(\U'(h_i){\bf \Psi}(x_i)+\W'(h_i))\Y'(x_i).\]
The quantity $\U'(h_i){\bf \Psi}(x_i)+\W'(h_i)$ is an $n\times n$ matrix which is already calculated in propagating ${\bf \Psi}$.
Note that $\Y'$ needs to be propagated only when computing an eigenfunction associated to an eigenvalue, i.e.\ only for the last energy value resulting from the shooting iteration process.

The eigenfunction \eqref{eigfu} can be multiplied by a scalar to achieve the appropriate eigenfunction normalization. We have 
\[\int_a^by^T_k(x)y_k(x)dx=w_L^T\left(\int_a^c\Y_L^T(x)\Y_L(x)dx\right) w_L+w_R^T\left(\int_c^b\Y_R^T(x)\Y_R(x)dx\right) w_R\]
which means that we need to compute the two integrals in the right hand side.
It can be shown that (see \cite{Ixaru3} for the scalar case)
 \[\int_a^c\Y_L^T(x)\Y_L(x)dx=\Y_L'^T(c){\bf \Psi}_L^{(E)}(c)\Y_L'^T(c)\]
and 
 \[\int_c^b\Y_R^T(x)\Y_R(x)dx=-\Y_R'^T(c){\bf \Psi}_R^{(E)}(c)\Y_R'^T(c).\]
 Thus the knowledge of the first derivative of ${\bf \Psi}$ with respect to $E$, allows the computation of the eigenfunction norm. As mentioned earlier,  this derivative can be simultaneously propagated by the CP algorithm.
%
\subsection{Stepsize selection and error control}
A very important property of the CP methods is that their errors are bounded with respect to the energy $E$. An implication is that the mesh needs to be generated only once (before the actual shooting process) and can be used to generate multiple eigenvalue approximations.
An automatic stepsize selection for CP methods applied on systems of Schr\"odinger equations has been presented in \cite{Ledoux2}. It constructs a mesh with non-equal steps whose lengths are consistent with a preset tolerance $tol$ by controlling the local error in the propagation matrix. A similar procedure is used here where we use all the terms in the formulae for the sixth order algorithm which are supplementary to the terms to be used in a (weaker) fourth order method, to construct the local error estimate. 

The matlab package also returns an estimation of the error in each eigenvalue approximation. These error estimates are obtained by calculating for each eigenvalue an associated 'reference'
eigenvalue on the reference mesh. The error estimate is then the difference between the eigenvalue and the more accurate reference
eigenvalue. 

\section{Some numerical results}
As a first testproblem for the automatic eigenvalue computation and the stepsize selection procedure, we consider the problem from \cite{IxaruE} which has known eigenvalues and solutions. The $2\times 2$ matrix potential for $x\geq 0$, is given by
\begin{align}
&V_{1,1}(x)=V_{2,2}(x)=V_{\rm PT}(x;45,1)+V_{\rm PT}(x;39/2,1/2)\label{Ix1}\\
&V_{1,2}(x)=V_{2,1}(x)=V_{\rm PT}(x;45,1)-V_{\rm PT}(x;39/2,1/2)\label{Ix2}
\end{align}
where $V_{\rm PT}$ is the P\"oschl-Teller potential
\[V_{\rm PT}(x;\nu,\alpha)=-\nu/\cosh^2(\alpha x).\] The wavefunction tends to zero when $x\to+\infty$. In our experiments the infinite integration interval $x\geq 0$ has been cut at $x=30$.
\begin{table}[!h]
	\begin{center}
	\caption{Some results obtained for the first testproblem when applying the procedure with two different input tolerances $tol$. The real absolute error and the estimated error are shown. $E_k$ is the exact eigenvalue and $nsteps$ is the number of steps in the mesh generated by the stepsize selection algorithm.}
	\label{tab:problem0}
		\begin{tabular}{llllll}
		\hline
			$k$&$E_k$&\multicolumn{2}{c}{$tol=10^{-6}$}&\multicolumn{2}{c}{$tol=10^{-9}$}\\
			&&Err&ErrEst&Err&ErrEst\\
			\hline
			0&$-64$&2e-7&2e-7&3e-10&3e-10\\
			1&$-36$&1e-6&1e-6&8e-10&8e-10\\
			2&$-30.25$&1e-8&1e-8&5e-12&5e-12\\
			3&$-20.25$&2e-8&2e-8&2e-11&2e-11\\
  		4&$-16$&2e-6&2e-6&1e-9&1e-9\\
  		5&$-12.25$&5e-8&5e-8&2e-11&2e-11\\
  		6&$-6.25$&9e-8&9e-8&3e-11&3e-11\\
  		7&$-4$&1e-6&1e-6&7e-10&7e-10\\
  		8&$-2.25$&1e-7&1e-7&4e-11&4e-11\\
  		9&$-0.25$&1e-7&1e-7&4e-11&4e-11\\
  		\hline
  		nsteps&&68&&263&\\
  		\hline
		\end{tabular}
	\end{center}
\end{table}
Table \ref{tab:problem0} shows the results obtained by the Matlab package for two different user input tolerances $tol$. One can observe that the generated eigenvalue approximations have accuracies which are in agreement with $tol$ and that the error estimate is adequate. 

As a second testproblem, we consider the following system of coupled differential equations which originates from applying separation of variables to the Schr\"odinger equation in ${\mathbb R}^2$ \cite{Marletta}
\begin{equation}\frac{d^2 y_i}{dx^2}=\sum_{j}[V_{ij}(x)-E\delta_{ij}]y_j(x).\label{testprob2}\end{equation}
When we choose the set of orthonormal functions 
\[\Phi_j(\theta)=\begin{cases}(2\pi)^{-1/2}, &j=0\\
\pi^{-1/2}\sin((j+1)\theta/2), &j\; {\rm odd}\\
\pi^{-1/2}\cos(j\theta/2), &j\; {\rm even},
\end{cases}\]
and the potential ${\hat V}(x,\theta)$ of the ${\mathbb R}^2$ problem of the form 
\[{\hat V}(x,\theta)={\bar V}(x)\pi^{-1/2}\left(\frac{1}{\sqrt{2}}+\frac{2\cos\theta-1}{(2\cos\theta-1)^2+\sin^2\theta}\right)={\bar V}(x)\sum_{k=0}^\infty\frac{1}{2^k}\cos(k\theta),\]
the matrix $\V$ in \eqref{testprob2} is then given by
\[V_{ij}(x)={\bar V}(x) Q_{ij}\]
with
\[ Q_{i+1,j+1}=\sum_{k=0}^\infty\frac{1}{2^k} \int_0^{2\pi} \Phi_i(\theta)\Phi_j(\theta) \cos(k\theta)d\theta.\]
The entries of this matrix ${\bf Q}$ can be evaluated explicitly. The vector $y$ is truncated to have $n$ elements and for the function ${\bar V}(x)$ we choose
a Woods-Saxon potential
\[{\bar V}(x)=-50\left(\frac{1-\frac{5t}{3(1+t)}}{1+t}\right),\quad t=e^{(x-7)/0.6}.\]
The boundary conditions imposed on the wave functions $y_j$ are
\[y_j(0)=0,\quad y_j(x_{\rm max})=0,\quad 0\leq x\leq x_{\rm max}=15.\]
\begin{figure}
	\centering
			\includegraphics[width=0.45\textwidth]{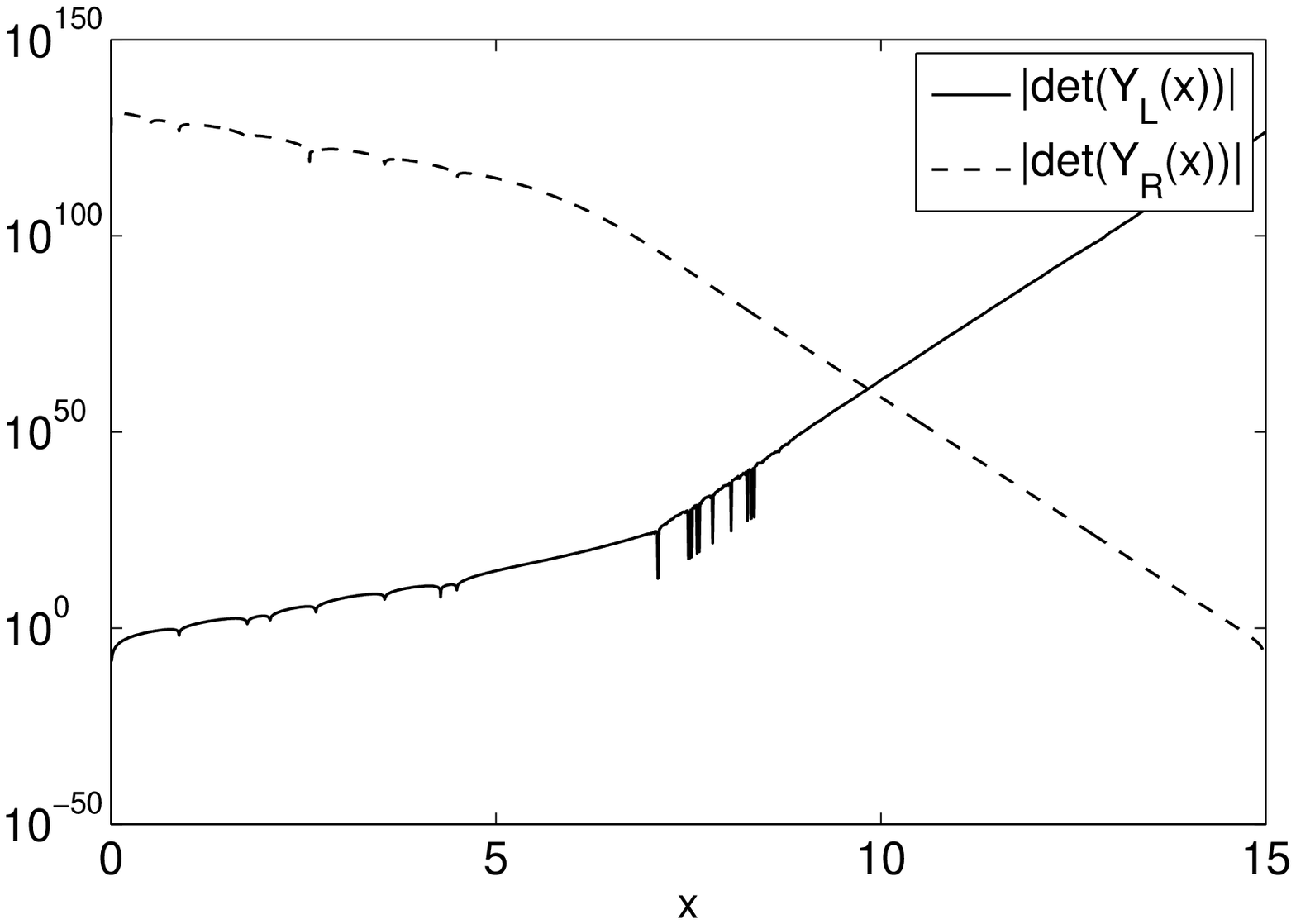}
			\includegraphics[width=0.45\textwidth]{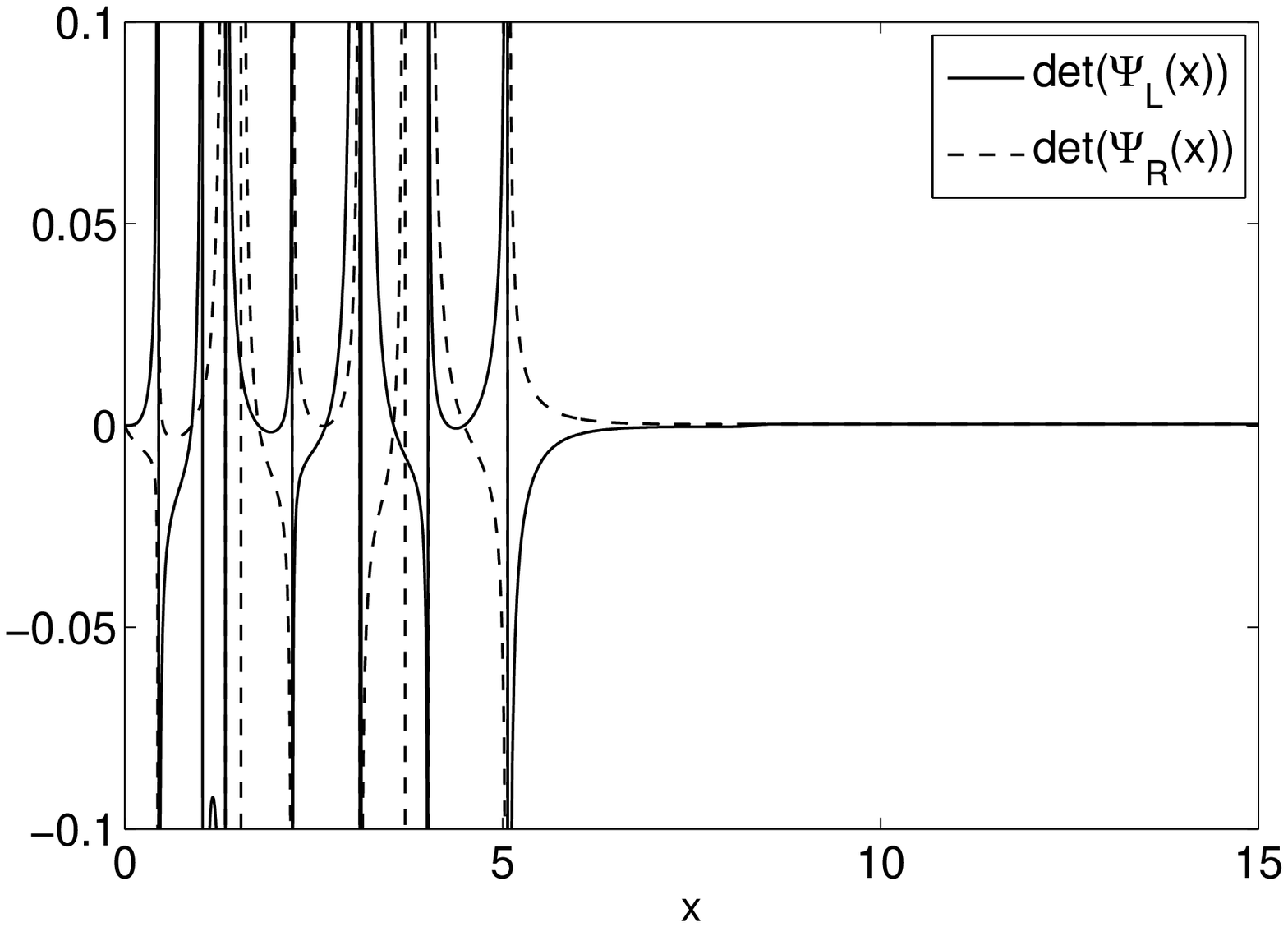}
	\caption{Illustration of the unstable propagation of $\Y$ and the stable propagation of ${\bf \Psi}$ for the second testproblem with $n=4$ and $E=E_8$.}	\label{fig:WS2}
\end{figure}
For some $E$-values in the spectrum some channels will be closed and others will be open over part of the domain.
The closed channels cause numerical instabilities when using a wavefunction propagation algorithm. This is illustrated in 
Figure \ref{fig:WS2} which shows the determinant of the matrix functions $\Y_L$, $\Y_R$, ${\bf \Psi}_L$ and ${\bf \Psi}_R$ for $n=4$ and $E$ equal to the eigenvalue $E_8=-53.43922418$. A CP scheme was used to propagate $\Y_L$ and ${\bf \Psi}_L$ outwards from $x=0$ to $x=15$ with initial conditions $\Y_L(0)={\bf 0},\Y'_L(0)={\bf I}$ and $\Y_R$ and ${\bf \Psi}_R$ inwards from $x=15$ to $x=0$ with initial conditions $\Y_R(15)={\bf 0},\Y'_R(15)={\bf I}$. Where $\det({\bf \Psi}_L)$ and $\det({\bf \Psi}_R)$ vanish in both endpoints and consequently match in both directions onto the boundary conditions, this is clearly not the case for $\det({\bf Y}_L)$ and $\det({\bf Y}_R)$.

We calculated the eigenvalues $E_0$ to $E_{25}$ for $n=4$ and $n=8$ at a tolerance of $10^{-6}$ using the procedure from section 5. The results are shown in Table \ref{tab:problem2}. Only the correct digits are shown, i.e.\ the ones which agree with the digits obtained with a $tol=10^{-12}$ computation.
Some values obtained with eigenvalue calculations based on finite difference discretization are also provided. The finite difference method with $nsteps$ mesh intervals results in a $n (nsteps-1)\times n (nsteps-1)$ matrix eigenvalue problem. One Richardson extrapolation step was used. The finite difference computations with one extrapolation step then results in $nsteps+nsteps/2$ evaluations of the potential matrix. The sixth order coefficient approximation method uses 3 potential matrix evaluations per mesh interval.
\begin{table}[!h]
	\begin{center}
	\caption{Eigenvalue computations for the second testproblem, obtained with the CP procedure with input tolerance $tol=10^{-6}$ and finite differences (FD). }
	\label{tab:problem2}
		\begin{tabular}{lllll}
		\hline
			$k$&\multicolumn{2}{c}{$n=4$}&\multicolumn{2}{c}{$n=8$}\\
			&CP&FD&CP&FD\\
			\hline
			0&-65.42657004&-65.427&-82.43582467&-82.44\\
			1&-64.03484348&-64.03&-80.97081456&-80.97\\
			2&-62.0689567&-62.07&-78.90949840&-78.9\\
			3&-59.61523778&-59.62&-76.3408597&-76.3\\
  		4&-56.7257918&-56.73&-73.3178863&-73.3\\
  		5&-55.1994967&-55.20&-71.98401865&-71.98\\
  		10&-49.5863494&-49.59&-66.0559592&-66.1\\
  		25&-32.0936608&-32.1&-43.9683184&-44.0\\
  		\hline
  		nsteps&72&4000&76&2000\\ 
  		\hline
		\end{tabular}
	\end{center}
\end{table}

The results generated by the implemented code show the power of our generalized Pr\"ufer representation, and (once again) the practical usefulness of CP methods. 


\section{Conclusion}
We discussed a shooting algorithm which approximates the eigenvalues and eigenfunctions of the coupled-channel Schr\"odinger equation. The method is based on the approximation of the potential matrix function and the quantity that is propagated is the $R$-matrix. The exponential build up of the wavefunction in the classically forbidden regions is therefore cancelled, and the method is inherently stable. Among the novel results presented in the paper, we can mention the observation that CP-based propagation algorithms can accurately pass through singularities in the flow associated to the $R$-matrix and that changing coordinates is consequently not required. We also introduced new algorithms to derive the node count from the knowledge of the $R$-matrix. This node count is immensely valuable when trying to find all the eigenvalues of the coupled equations that lie in a particular energy range or with a specific index. Supplemented with a stepsize selection algorithm and error control, the proposed algorithms were shown to be very well suited for the implementation in automatic software.

\end{document}